\documentclass [reqno,11pt,a4paper,fullpage]{amsart}
\usepackage{amsmath}%
\usepackage{amsfonts}%
\usepackage{amssymb}%
\usepackage{graphicx}
\usepackage{multirow}
\usepackage{enumitem}
\usepackage{a4wide}
\usepackage{natbib}
\usepackage{cite}
\usepackage[breaklinks]{hyperref}

\newtheorem{theorem}{Theorem}
\newtheorem{lemma}{Lemma}
\begin{document}
\author{C. Reisinger}
\address
{Mathematical Institute, University of Oxford, 24-29 St Giles', Oxford, OX1 3LB, U.K.}
\email{reisinge@maths.ox.ac.uk}%
\author{A. Whitley}
\email{whitley@maths.ox.ac.uk}%

\title[A natural time change for the heat equation]{The impact of a natural time change on the convergence of the Crank-Nicolson scheme}



\begin{abstract}
{We first analyse the effect of a square root transformation to the time variable on the convergence of the Crank-Nicolson scheme when applied to the solution of the heat equation with Dirac delta function initial conditions. In the original variables, the scheme is known to diverge as the time step is reduced with the ratio, $\lambda$, of the time step to space step held constant and the value of $\lambda$ controls how fast the divergence occurs. After introducing the square root of time variable we prove that the numerical scheme for the transformed partial differential equation now always converges and that $\lambda$ controls the order of convergence, quadratic convergence being achieved for $\lambda$ below a critical value. Numerical results indicate that the time change used with an appropriate value of $\lambda$ also results in quadratic convergence for the calculation of the price, delta and gamma for standard European and American options without the need for Rannacher start-up steps.}\\\
\smallskip
\noindent \textit{Keywords.}
{Heat equation; Crank-Nicolson scheme; convergence; Black-Scholes; European option; American option; asymptotics; time change.}
\end{abstract}

\maketitle

\section{Introduction}

The Crank-Nicolson scheme is a popular time stepping scheme for the numerical approximation of diffusion equations and is particularly heavily used for applications in computational finance.
This is due to a combination of favourable properties: its simplicity and
ease of implementation (especially in one dimension);
second order accuracy in the timestep for sufficiently regular data;
unconditional stability in an $L_2$ sense.

However, there are well-documented problems which arise from the fact that the scheme is at the
cusp of stability in the sense that it is \emph{A-stable} \citep[see][]{S} but does not dampen so-called stiff components effectively (i.e., is not \emph{strongly A-stable}). Specifically, a straightforward Fourier analysis of the scheme applied to a standard finite difference discretisation of the heat equation shows that high wave number components are reflected in every time step but asymptotically do not diminish over time.
This gives rise to problems for applications with non-smooth data, where the behaviour of these components over time plays a crucial role in the smoothing of solutions.
Examples where this is relevant include Dirac initial data, as they appear naturally for adjoint equations \citep[see][]{GS}, and piecewise linear and step function payoff data in the computation of option prices \citep[see][]{WHD}.

There, the situation is exacerbated if sensitivities of the solution to its input data (so-called `Greeks') are needed; see \citet{SHAW} for an early discussion of issues with time stepping schemes in this context.

There is a sizeable body of literature concerned with schemes with better stability properties, in particular \emph{L-stable} schemes \citep[see][]{S} which share with the underlying continuous equation the property that high wave number components decay increasingly fast in time.
Examples for schemes with this property include sub-diagonal Pad\'e schemes based on rational approximations of the exponential function.
In \citet{RANNACHER}, Rannacher proposes to replace the first $2 \mu$ steps of a higher-order Pad\'e scheme of order $(\mu,\mu)$ (e.g., the Crank-Nicolson scheme for $\mu=1$) by a low-order sub-diagonal Pad\'e scheme with order $(\mu-1,\mu)$ (e.g., backward Euler for $\mu=1$), and shows that this restores optimal order for diffusion equations; the case $\mu=1$ is already analysed in \citet{LR}.

It is demonstrated in \citet{PVF} how this procedure can be used to obtain stable and accurate approximations to option values and their sensitivities for various non-smooth payoff functions.
This has been widely adopted in financial engineering practice.
\citet{KVY}, Khaliq
extend this to more general Pad\'e schemes and demonstrate numerically the stability and accuracy in practice for options with exotic payoffs.
Further adaptations are possible, for instance, where \citet{WKYVD}
give an application to discretely sampled barrier options, where discontinuities are introduced at certain points in time and a new `start-up' is needed (see already \citet{RANNACHER} for the provision of such restarts in an abstract context).

\citet{GC} analyse the behaviour of the Crank-Nicolson numerical scheme when used to solve a convection-diffusion equation with constant coefficients and Dirac delta function initial conditions. The numerical solution obtained by this method diverges as the time step goes to zero with the ratio $\lambda=k/h$ of the time step, $k$, to the space step, $h$, held constant. Their analysis in the frequency domain shows that there are errors associated with high wave numbers which behave as $O(h^{-1})$ and which will eventually overwhelm the $O(h^2)$ errors associated with low wave numbers as the time step goes to zero, thus explaining why the errors in this scheme will eventually increase as the time step is reduced. Keeping the ratio $\lambda$ constant is desirable because the Crank-Nicolson central difference scheme is second order consistent in both space and time and the numerical scheme is more efficient if this property is exploited.  
The authors extended their analysis to show why the incorporation of a small number of initial fully implicit time steps (i.e., the Rannacher scheme, \citet{LR,RANNACHER}) eliminates this divergence.

In this paper, an alternative method of avoiding the divergence is proposed and analysed. The idea is to introduce a time change into the partial differential equation (PDE) by transforming the original time variable, $t$, to
\[\tilde t =\sqrt{t}\]
\noindent and solving the equation numerically using the Crank-Nicolson scheme in the new variables.
The time change will be applied to the heat equation (\ref{HE}) on $\mathbb{R}\:\times\:[0,T]$ and can be considered natural as the heat equation with suitable initial data admits similarity solutions which depend on $x/\sqrt{t}$.
From a probabilistic angle, the heat equation is the Kolmogorov forward equation for the transition density of Brownian motion whose standard deviation at time $t$ is $\sqrt{t}$.\\
\indent The main result of the paper is given by the following Theorem.
\begin{theorem}\label{T:MainTheorem}
The Crank-Nicolson central difference scheme with uniform time steps exhibits second order convergence (in the maximum norm) for the time-changed heat equation, with Dirac initial data, as the time step $k$ tends to zero with $\lambda=k/h$ held fixed, provided that $\lambda \leq 1/\sqrt{2}$. For $\lambda \geq 1/\sqrt{2}$ the scheme is still convergent with order $1/\lambda^2$.
\end{theorem} 

The peculiar dependence of the convergence rate on the mesh ratio is in fact seen to be sharp up to a logarithmic factor. It also follows from the analysis of the heat equation that the computational complexity for an optimal choice of the mesh ratio is lower than for the Rannacher scheme in its optimal refinement regime.

As an added benefit, we obtain experimental second order convergence for the value, delta and gamma of European and American options.
\citet{FV} observe that at-the-money prices of European and American options computed by a Crank-Nicolson finite difference scheme exhibit a reduced convergence order of 1 in the time step $k$, which improves to 2 for Rannacher start-up in the case of European options, but only to 3/2 for American options.
This last observation is rationalised
by the square-root behaviour of the value and exercise boundary for short time-to-expiry. A heuristic adaptive time-stepping scheme based on this observation is shown there to restore second order convergence.
We show here, by means of numerical tests, that the square root time change introduced above provides a similar remedy even without Rannacher start-up. This is not surprising bearing in mind the relation between the time change and a non-uniform time mesh in the original time variable.
Precisely,
the time-changed scheme with constant steps of size $\sqrt{T}/N$, where $N$ is the total number of time steps for time up to $T$,
is equivalent to a non-uniform scheme in the original time variable with time points $t_m=(km)^2$ for $m=0,...,\sqrt{T}/k$.
The step size hence increases in proportion to $2 \sqrt{t}$ and the smallest (the first) step is of size $T/N^2$.


The remainder of this article is organised as follows. In Section~\ref{S:CNScheme}, we apply the time change transformation to the heat equation, describe the transformed numerical scheme, and calculate the discrete Fourier transform of the numerical solution.
Then, in Section~\ref{S:AsympFT}, we perform asymptotic analysis of the error between the transform and the transform of the true solution, identifying four wave number regimes which will then be used in the later analysis.
In Section~\ref{S:AsympErrors}, we use the results of the asymptotic analysis to determine the asymptotic behaviour of the errors between the numerical solution and the true solution, from which we deduce the convergence behaviour of the transformed scheme and hence Theorem~\ref{T:MainTheorem}.
Section~\ref{sec:numerical} contains numerical results illustrating these findings and compares the computational complexity to the Rannacher scheme.
In Section~\ref{S:BSApp}, the time change transformation is applied to the Black-Scholes equation and the solution method is used to calculate the gamma of a European call option. The behaviour of the resulting errors is described and explained in terms of the values of the strike and volatility.
Finally, in Section~\ref{S:AmericPut}, the time change transformation is applied to a penalty method used to calculate the price of an American put, as described in \citet{FV}.
Section~\ref{sec:concl} concludes.
\section{The Crank-Nicolson scheme applied to the transformed heat equation and its behaviour in the Fourier domain}\label{S:CNScheme}
To simplify the analysis, we restrict attention to the heat equation (i.e., we do not consider any convection terms)
\begin{equation}
u_t= \frac{1}{2} \, u_{xx},
\label{HE}
\end{equation}

\noindent for $t \in [0,T]$, $x \in \mathbb{R}$, where the solution, $u$, satisfies $u(x,0)=\delta(x)$, where $\delta$ is the Dirac delta function and which, after the change of time variable, becomes
\[u_{\tilde t}= \tilde t\: u_{xx},\]
\noindent where $\tilde t = \sqrt{t}$.

\noindent The transformed equation is then discretised using the Crank-Nicolson method, leading to the following numerical scheme

\begin{equation}
\frac{U^{n+1}_j-U^{n}_j}{k}=\frac{1}{2}\: \tilde t^{n}\:\frac{U^{n}_{j+1}-2U^{n}_j+U^{n}_{j-1}}{h^2}+\frac{1}{2}\: \tilde t^{n+1}\:\frac{U^{n+1}_{j+1}-2U^{n+1}_j+U^{n+1}_{j-1}}{h^2},
\label{EQ1}
\end{equation}

\noindent where $U^{n}_j$ is the numerical solution at the spatial grid node $j$ at the `time' step $n$, where the space grid, in principle, extends from $-\infty$ to $+\infty$, with $x_j=jk$ for $j \in \mathbb{Z}$. In the transformed `time' direction, we have  $\tilde t^i=ik$ for $1 \leq i \leq N$ where $N$ is the number of timesteps. We can simplify this to
\[(1+(n+1)\lambda^2)U^{n+1}_j-\frac{1}{2}(n+1)\lambda^2U^{n+1}_{j+1}-\frac{1}{2}(n+1)\lambda^2U^{n+1}_{j-1}=(1-n\lambda^2)U^{n}_j+\frac{1}{2}n\lambda^2U^{n}_{j+1}+\frac{1}{2}n\lambda^2U^{n}_{j-1}, \]

\noindent where $\lambda=k/h.$

To study the behaviour of this scheme in the Fourier domain we apply the discrete Fourier transform to this equation (e.g., see \citet{GC} and \citet{STRANG}). 
Multiplying equation~\eqref{EQ1} by $e^{i s x_j}$, summing over $j$ and simplifying, we obtain the following recurrence for the Fourier transform 
\[\widehat{U}^n(s)=h \sum ^{j=+\infty}_{j=-\infty} U_j^n e^{i s x_j}\]
at successive time steps,
\[\left(1+2(n+1)\lambda^2 \sin^2\left(\frac{sh}{2}\right)\right)\widehat{U}^{n+1}(s)=\left(1-2n\lambda^2 \sin^2\left(\frac{sh}{2}\right)\right)\widehat{U}^{n}(s).\]
We note that $\widehat{U}^0(s)\equiv1$, as the initial condition is a delta function at $(0,0)$, so we have
\begin{equation}
\label{EQ2}
\widehat{U}^{N}(s)=\frac{1(1-\xi)(1-2\xi) \ldots (1-(N-1)\xi)}{(1+\xi)(1+2\xi) \ldots (1+N\xi)},
\end{equation}

\noindent where $\xi = 2 \lambda^2 \sin^2(sh/2)$ and we want to study the behaviour of $\widehat{U}^{N}(s)$ as $N \rightarrow \infty$ with $kN$ and $k/h$ held fixed.\
\newline 
\indent In the subsequent analysis, it is often simpler to describe a condition on $s$, by stating the related condition on $\xi$, but it should be noted that the relationship between $s$ and $\xi$ depends on $h$.\

\indent In what follows, and to simplify the notation, we will fix $T=1$, so that
\[N=\frac{1}{k}=\frac{1}{h \lambda}.\] 
We note that the exact solution of the PDE is given by

\[u(x,t)=\frac{1}{\sqrt{2 \pi t}}\exp\left(-\frac{x^2}{2t}\right) \]

\noindent and its Fourier transform is given by $\widehat{u}(s,t)=\exp\left(-s^2 t/2\right)$. So, for                       $t=\tilde t=1$, we have $\widehat{u}(s,1)=\exp\left(-s^2/2\right)$.

As in \citet{GC} we can estimate the errors in the numerical solution (at time $t=\tilde t=1$), i.e. the differences between the values of $U^{N}_j$ and  $u(x_j,1)$, by applying the inverse Fourier transform to $\widehat{U}^N(s)-\widehat{u}(s,1)$.

%
%

\section{Asymptotic analysis of the Fourier transform}\label{S:AsympFT}

In this section we identify four wave number regimes for the Fourier transform and obtain useful asymptotic estimates for $\widehat{U}^N(s)$, in each of these regimes. Initially, we start by identifying three regimes, the low, intermediate and high wave number regimes (as in \citet{GC}), but then find it useful to subdivide the high wave number regime into two further regimes. 
\newline
\indent In essence, the wave number regimes are determined by the locations of the real zeros of the equation

\[\widehat{U}^{N}(s)=0,\]

\medskip
\noindent which are at $s_m$, for $m=m^*,\ldots,N-1$, for some integer $m^* \geq 1$, which depends only on $\lambda$, where
\[2 \lambda^2 \sin^2\left(\frac{s_mh}{2}\right)=\frac{1}{m}.\]
\noindent and we note that $m^* \geq \frac{1}{2\lambda^2}$.
Note that each $s_m$ is a function of $h$ and $s_m \rightarrow \infty$ as $h \rightarrow 0$, so these wave number boundaries increase in absolute terms while always being less than or equal to $\pi/h$.

The low wave number regime (which we will refer to as regime I) is, in part, defined by the requirement that all the terms in the numerator of the expression for $\widehat{U}^{N}(s)$ in equation (\ref{EQ2}) should be positive. This requires that $\xi < 1/(N-1)$. If we assume the stronger condition that $\xi < 1/N$, we can write a truncated Taylor expansion for

\[\log(\widehat{U}^N(s))=\sum^{N-1}_{m=1} \log(1-m\xi) - \sum^{N}_{m=1} \log(1+m\xi) \]

\noindent from which we can later derive an asymptotic estimate for $\widehat{U}^N(s)$.

As in \citet{GC}, the low wave number regime is further restricted by the condition that $s<h^{-r}$, for a value $r$ such that  $r<\frac{1}{3}$, which ensures that the remainder terms in the Taylor expansion go to zero as $h \rightarrow 0$, so that we can derive a useful approximation for $\widehat{U}^N(s)-\widehat{u}(s,1)$. As this condition also ensures that $\xi < 1/N$ asymptotically, this condition alone is sufficient to define the low wave number regime.

Next, there exists an intermediate wave number regime (regime II) defined by the range of wave numbers for which $\xi < 1/N$ but which are not in regime I (cf. \citet{GC}).

The high wave number regime is defined by the requirement that $\xi \geq 1/N$. If this condition on $\xi$ is satisfied then some of the terms in the numerator of the expression for $\widehat{U}^{N}(s)$ in equation (\ref{EQ2}) will be negative. There will be an integer $m > 0$ such that $\xi \in [1/(m+1),1/m]$ and the number of positive terms will remain fixed as $N$ increases while the number of negative terms will increase with $N$. A particular negative factor $(1-r\xi)$ can be rewritten as $-r\xi(1-1/r\xi)$, with $1/r\xi<1$, and we can then write a truncated Taylor expansion for $\widehat{U}^{N}(s)$ in terms of $1/\xi$ which will then lead to an asymptotic value for $\widehat{U}^{N}(s)$.

At first sight, it might seem necessary to treat separately all the $\xi$-intervals of the form, $[1/(m+1),1/m]$, giving rise to an ever-growing set of wave number regimes but, in fact, it suffices for the results we wish to prove, to treat the last of these intervals, $[1/m^*,2\lambda^2]$, separately and combine all the other intervals into a single regime.

So we define wave number regime III to correspond to the interval $[1/N,1/m^*]$ and wave number regime IV to correspond to the interval $[1/m^*,2\lambda^2]$. These ranges are equivalent to $s_N \leq s \leq s_{m^*}$ and $s_{m^*} < s \leq \pi/h$ respectively. We will obtain a uniform bound on the magnitude of $\widehat{U}^{N}(s)$ for $\xi \in [1/(m+1), 1/m]$ in regime III and an asymptotic value for the magnitude of $\widehat{U}^{N}(s)$ in the regime IV.

Table ~\ref{tab:TableReg} summarises the wave number regimes giving their boundaries in terms of both $\xi$ and $s$. (Only the regimes for positive wave numbers are shown. By symmetry, the regimes for the negative wave numbers are just the negatives of these intervals). 
\begin{table}[thb]
\centering
{
\renewcommand{\arraystretch}{1.5}
\begin{tabular}{|l|l|l|l|l|}
\hline
\multicolumn{5}{|c|}{Wave number regimes}\\
\hline
Regime & $\xi_{min}$ & $\xi_{max}$ & $s_{min}$ & $s_{max}$ \\ \hline
\ I & 0 & $\xi_r $ & 0 & $ h^{-r} $ \\
\ II & $\xi_r$ & $\frac{1}{N}$ & $ h^{-r} $ & $s_{N}$ \\
\ III & $\frac{1}{N}$ & $\frac{1}{m^*}$ & $s_{N}$ & $s_{m^*}$ \\
\ IV & $\frac{1}{m^*}$ & $2\lambda^2$ & $s_{m^*}$ & $\frac{\pi}{h}$ \\ \hline
\end{tabular}\\
}
\medskip
\caption{Limits of wave number regimes in terms of $\xi$ and $s$.
Here, $r<{1}/{3}$, $\xi_r=2\lambda^2\sin^2({h^{1-r}}/{2})$, $s_n$ are defined by ${1}/{n}=2\lambda^2\sin^2({s_n h}/{2})$ for $n=N$ and $n=m^*$,
and $m^*\ge 1$ is the smallest $n$ for which there exists such $s_n$}
\label{tab:TableReg}
\end{table}
The results of carrying out asymptotic expansions in the four wave number regimes are summarised in the following Lemma which gives the behaviour of the error between the two transforms, $\widehat{E}(s)=\widehat{U}^N(s)-\widehat{u}(s,1)$.

\begin{lemma}\label{L:LemmaA}

The following formulae give the errors in the Fourier transform in each wave number regime (see Table \ref{tab:TableReg}).\

\medskip 

\begin{enumerate}[leftmargin=1cm,rightmargin=0cm]

\item
\noindent In wave number regime I,

\[\widehat{E}(s)=\widehat{u}(s,1)\left(\frac{1}{24}s^4-\frac{1}{48}\lambda^2 s^6+\frac{1}{8}\lambda^2 s^4\right)h^2(1+o(1)).\]\\
\item
\noindent In wave number regime II,
\[\widehat{E}(s)=o(h^q)\]

\noindent for any $q > 0$.

\medskip 
\item
\noindent In wave number regime III,

\[|\widehat{E}(s)| \leq	\frac{((m+1)!)^2(N-m-1)!}{(2m+2)(N+m)!}\]

\noindent whenever $\xi=2\lambda^2\sin^2(\frac{sh}{2}) \in [1/{(m+1)},1/{m}]$ and, everywhere in this regime, 
\[
\widehat{E}(s) = O(h^{2m^* + 1})
\]
for some $m^*\ge 1$.

\item
\noindent Finally, in wave number regime IV, there exists a positive, continuous, and hence bounded, function $S(\cdot,m^*)$ on $[1/m^*,2\lambda^2]$, such that 

\[|\widehat{E}(s)| = S(\xi,m^*)h^{(1+\frac{2}{\xi})}(1 + O(h))). \]

\end{enumerate}

\end{lemma}

\begin{proof} The key features of the proof are given here while further details are included in the Appendix.\\

\indent We first consider wave number regime I.\
Inspection of equation \eqref{EQ2} shows that all the terms in the numerator are positive if $\xi \leq 1/N<1/(N-1)$. We can then write $\log(U^N(s))$ as
\begin{equation}
\label{EQ3}
\log(\widehat{U}^N(s))=\sum^{N-1}_{k=1} \left(\log(1-k\xi) - \log(1+k\xi)\right) - \log(1+N\xi)
\end{equation}

\noindent and expand the component terms, $\log(1 \pm k\xi)$, by first expanding $\xi$ in powers of $h$. The calculations are very similar to those in \citet{GC} and the details of the expansions are left to the Appendix.
We find that

\begin{equation}
\label{expapp}
\log(\widehat{U}^N(s))=-\frac{1}{2}s^2 +\left(\frac{1}{24}s^4-\frac{1}{48}\lambda^2 s^6+\frac{1}{8}\lambda^2 s^4\right)h^2+O(s^8h^3).
\end{equation}

\noindent we find we require that the second term converges to zero as $h \rightarrow 0$ and the last term is asymptotically dominated by the second term. These conditions are satisfied by requiring that $s<h^{-q}$ with $q < \frac{1}{3}$ and this condition on $s$ defines the upper limit of the wave number regime I.

We then have, with $\widehat{u}(s,1)=e^{-\frac{1}{2}s^2}$,

\[\widehat{U}^N(s)-\widehat{u}(s,1)=\widehat{u}(s,1)\left(\frac{1}{24}s^4-\frac{1}{48}\lambda^2 s^6+\frac{1}{8}\lambda^2 s^4\right)h^2(1+o(1))\]

\noindent as we set out to prove.\\

\indent We now consider wave number regime II.\\

\noindent The analysis is again similar to that in \citet{GC}. The lower limit of wave number regime II is $O(h^{-\frac{1}{3}})$ (which corresponds to the upper limit of wave number regime I) and the upper limit of the regime corresponds to the first zero of $\widehat{U}^N(s)$, i.e., $s=(2/h)\sin^{-1}\sqrt{(1/2 \lambda^2(N-1))}$, which is $O(h^{-\frac{1}{2}})$. In this regime each of the factors $(1-m\xi)/(1+(m+1)\xi)$ in $\widehat{U}^N(s)$ is positive and decreasing as $s$ increases.\\ 
\indent Consequently, the value of $\widehat{U}^N(s)$ will be less than or equal to its value at $s=h^{-r}$ for $0<r<1/3$.
However, at $s=h^{-r}$ we have, for any $q > 0$,

\[\frac{\widehat{U}^N(s)}{h^q}= \frac{   e^{  -\frac{1}{2} h^{-2r}}   }{h^q} \exp{\left(h^{2(1-2r)}\left(\frac{1}{24}+\frac{\lambda^2}{8}\right)-h^{2(1-3r)}\frac{1}{48}\lambda^2+ o(1)\right)}\rightarrow 0\]
\noindent as $h \rightarrow 0$, so $\widehat{U}^N(s)=o(h^q)$ at $s=h^{-\frac{1}{3}}$
and so we have $\widehat{U}^N(s)=o(h^q)$ throughout this regime.\\

\indent Next we consider the high wave number regime which is split into two parts.\\
\noindent\\
\indent It is in wave number regimes III and IV that our analysis differs significantly from that in \citet{GC}
where there are repeated factors in the corresponding expression for $\widehat{U}^N(s)$, which results in a simpler expansion in terms of $1/\xi$. In our case, $\widehat{U}^N(s)$ is a product of many more factors, so the analysis and expansion is more complicated. However, we can obtain our key results by breaking the high wave number regime into two parts and we will find that we only need to perform an expansion in regime IV.\\

For wave number regime III, which extends from $s_N=(2/h)\sin^{-1}\sqrt{(1/2\lambda^2 N)}$ to $s_{m^*}$, corresponding to $\xi$ running from $1/N$ to $1/m^*$, with $m^* \geq 1$, we investigate the behaviour of $\widehat{U}^N(s)$ for values of $s$ corresponding to values of $\xi$ in the intervals $[1/(m+1), 1/m]$, where $m^* \leq m \leq N-1$.
We write $\widehat{U}^N(s)$ as
\begin{eqnarray}
\label{uhatn}
\widehat{U}^N(s)= \left\{  \frac{ 1(1-\xi)(1-2\xi)\ldots(1-m\xi) }{ (1+\xi)(1+2\xi)\ldots(1+(m+1)\xi) } \right\} . \left\{\frac{ (1-(m+1)\xi)\ldots(1-(N-1)\xi) }{ (1+(m+2)\xi)\ldots(1+N\xi) }\right\}.
\end{eqnarray}
For the range of values of $\xi \in [1/(m+1), 1/m]$ being considered, the left-hand expression is a product of non-negative factors, $(1-\nu\xi)/(1+(\nu+1)\xi)$, for $\nu = 0,...,m$, each of which is decreasing in $\xi$, so the left-hand expression is bounded by the value of the expression evaluated at $\xi=1/(m+1)$, i.e., $((m+1)!)^2 /(2m+2)!$\\

Similarly, the right-hand expression is a product of factors $(1-\nu\xi)/(1+(\nu+1)\xi)$, for $\nu=m+1,...,N-1$,	 which are all negative and whose magnitude is increasing with $\xi$. So the magnitude of the right-hand expression is bounded by the magnitude of the expression evaluated at $\xi=1/m$, i.e. $(N-m-1)!(2m+1)!/(N+m)!$\\

Therefore, the magnitude of $\widehat{U}^N(s)$ for $1/(m+1) \leq \xi \leq 1/m$ is bounded by 

\begin{equation}
W_{N,m}= \frac{((m+1)!)^2(N-m-1)!}{(2m+2)(N+m)!},
\end{equation}

\noindent and we consider the behaviour of this expression as $m$ decreases from $N-1$ to $m^*$.
We have

\[\frac{ W_{N,m-1}}{ W_{N,m}}=\frac{N^2-m^2}{m(m+1)},\]

\noindent and this expression is less than one for

\[N^2-m^2 < m(m+1), \]

\noindent which corresponds approximately to $m > N/\sqrt{2}$. This implies that, as $m$ decreases from $N-1$, $W_{N,m}$ first decreases, reaching a minimum in the vicinity of $N/\sqrt{2}$. So the maximum of $|\widehat{U}^N(s)|$ will occur either at $m=N-1$ or at $m=m^*$.\\

We now compare the values of $W_{N,m}$ at these values of $m$, by calculating the ratio

\[\frac{ W_{N,m^*}}{ W_{N,N-1}}=\frac{((m^*+1)!))^2}{(2m^*+2)}\frac{(N-m^*-1)!}{(N+m^*)!}\frac{(2N)!}{(N!)^2}.        \]\\

\noindent By Stirling's formula, we have the approximation

\[\frac{(2N)!}{(N!)^2}\sim \frac{(2N)^{2N+\frac{1}{2}}e^{-2N}}{((N)^{N+\frac{1}{2}}e^{-N})^2}\sim \frac{2^{2N+\frac{1}{2}}}{N^{\frac{1}{2}}}    \]

\noindent and we see that the exponential term dominates the other powers of $N$ and $W_{N,m^*}/W_{N,N-1} \rightarrow \infty$ as $N \rightarrow \infty $. So, for large enough $N$, $W_{N,m}$ is maximised at $m=m*$ and it then follows that the expression

\[\frac{((m^*+1)!)^2}{(2m^*+2)}\frac{(N-m^*-1)!}{(N+m^*)!} = O(N^{-2m^*-1}) = O(h^{2m^*+1}) \]

\noindent gives a uniform bound on $|\widehat{U}^N(s)|$ in wave number regime III. \\
\noindent\\
\indent Finally, we consider wave number regime IV.\\
\noindent\\
In this regime, we will find that we must perform an expansion of $\widehat{U}^N(s)$ from (\ref{uhatn}) and we write

\[
\big|\widehat{U}^N(s)\big| = M(\xi,m^*)\left|\frac{ (1-(m^*+1)\xi)\ldots(1-(N-1)\xi) }{ (1+(m^*+1)\xi)\ldots(1+(N-1)\xi) }\right|(1+N\xi)^{-1},
\]

\noindent where

\[M(\xi,m^*) = \left|  \frac{(1-\xi)(1-2\xi)\ldots(1-m^*\xi) }{ (1+\xi)(1+2\xi)\ldots(1+m^*)\xi) } \right|, \]


\noindent so

\begin{equation}
\label{secexp}
\log\big(\big|\widehat{U}^N(s)\big|\big)=\log(M(\xi,m^*))+\log\left\{\frac{(1-\frac{1}{\xi(m^*+1)})\ldots(1-\frac{1}{\xi(N-1)})} {(1+\frac{1}{\xi(m^*+1)})\ldots(1+\frac{1}{\xi (N-1)})}\right\}+\log((1+N\xi)^{-1})
\end{equation}

\noindent and we seek an expansion of the left-hand side of this equation in terms of $1/\xi$ to find, as described in the Appendix, that\

%
%

%
\begin{eqnarray}
\label{expandapp}
\big|\widehat{U}^N(s)\big|=S(\xi,m^*)h^{({2}/{\xi}+1)}(1+O(h)),
\end{eqnarray}
\newline
\noindent where $S$ is a function of $\xi$ and $m^*$, which is positive and continuous on $[1/m^*,2\lambda^2]$, as we wanted to show.

\end{proof}

\section{Error contributions from the various wave number regimes}~\label{S:AsympErrors}

%
\indent The error in the numerical scheme at $x_j$ is found by using the inverse Fourier transform

\[E_j=\frac{1}{2\pi}\int ^{s=\frac{\pi}{h}}_{s=-{\frac{\pi}{h}}} \widehat{E}(s) \exp(-is x_j) \; ds, \]

\noindent where $\widehat{E}(s)$ is the error in the Fourier transform. So decomposing the error, applying the inverse  Fourier transform 
and using symmetry, we can write
\[E_j=\sum_{i=1}^{4}  E^i_j,\]
\noindent where
\begin{eqnarray}
\label{eij}
E^i_j=\frac{1}{\pi}\int _{s \in I_i}\widehat{E}(s) \cos(s x_j) \; ds,
\end{eqnarray}

\noindent where $I_i$ is wave number region for $i\in \{I,II,III,IV\}$. 
We then have the following Lemma which makes use of the results from Lemma ~\ref{L:LemmaA} in Section 3.

\begin{lemma}\label{L:LemmaB}

With $E^i = (E_j^i)_{j\in \mathbb{Z}}$ defined in (\ref{eij}), the contribution to the error of the numerical scheme from the four wave number regimes is given by
\begin{eqnarray*}
\begin{array}{llll}
E^{I} &=& O(h^2), &\\
E^{II} &=& o(h^q) &\quad \mathit{  for}\: \mathit{any  } \: q>0, \\
E^{III} &=& O(h^{2m^*}) &\quad \mathit{ for}\: \mathit{some} \: m^*\geq 1, \\
E^{IV} &=& O\left(\frac{h^\frac{1}{\lambda^2}}{\sqrt{\log{\frac{1}{h}}}}\right).&
\end{array}
\end{eqnarray*}

%
%
%
%
%
%
%
%
%
%
%
%
%
%

\end{lemma}

\begin{proof}

Again, only key features of the proof are given here while further details are included in the Appendix.\\

In wave number regime I, the analysis is similar to that in \citet{GC}. The key observation is that the term $\widehat{u}(s,1)(\frac{1}{24}s^4-\frac{1}{48}\lambda^2 s^6+\frac{1}{8}\lambda^2 s^4)$ has a readily identifiable inverse Fourier transform, and we deduce that the error contribution is $O(h^2)$.\\

In wave number regime II, we again follow the analysis in \citep{GC} to conclude that the error contribution from this regime will be dominated by $h^q$ for any $q > 0 $.\\

In wave number regime III, we can write

\begin{equation}
\label{EQ4}
\int _{s_{N-1}}^{s_{m^*}} \big|\widehat{U}^N(s) \cos(s x_j)\big| \; ds \leq \int _{0}^{\pi/h} \big|\widehat{U}^N(s)\big| \; ds,
\end{equation}


\noindent which is $\pi/h \cdot O(h^{2 m^* + 1})$, i.e., $O(h^{2m^*})$. \\ 

Finally, we consider the error contribution from wave number regime IV, i.e., where $s_{m^*} \leq s \leq \pi/h$.\\
We need to determine the behaviour, as $h \rightarrow 0$, of

\[\int^{\frac{\pi}{h}}_{s=s_{m^*}} S(\xi,m^*)h^{(1+\frac{2}{\xi})} \cos(s x_j)\; ds,\]

\noindent the magnitude of which is maximised at $j=0$ (this is because $S$ is positive in this wave number regime). So we only need to consider the value of 

\[ F(h)=\int^{\frac{\pi}{h}}_{s=s_{m^*}} S(\xi,m^*) h^{(1+\frac{2}{\xi})} \; ds.\]\\

\noindent From $\xi=2\lambda^2\sin^2(\frac{sh}{2})$, we find that

\[ F(h)=\int^{2\lambda^2}_{\xi=\xi_{m^*}} \frac{ S(\xi,m^*)h^{\frac{2}{\xi}} }{  \lambda\sqrt{2}\sqrt{\xi}\sqrt{1-\frac{\xi}{2\lambda^2}} } \; d\xi. \]

\noindent As the lowest (i.e., dominant) order in $h$ from $h^{\frac{2}{\xi}}$ occurs when $\xi=2\lambda^2$ we expect that this integral will be close to $O(h^{\frac{1}{\lambda^2}})$ and using careful, but elementary, analysis 
(see the Appendix) we find that

\begin{eqnarray}
\label{intapp}
I=\frac{F(h)}{h^{\frac{1}{\lambda^2}}}=O\left(\frac{1}{ \sqrt{   \log{   1/h   }    }}\right),
\end{eqnarray}

\noindent and so the error contribution in regime IV is found to be $O\left({h^{\frac{1}{\lambda^2}}}/{ \sqrt{   \log{   \frac{1}{h}   }    }}\right)$. 

\end{proof}

As a consequence of Lemma~\ref{L:LemmaB}, we obtain Theorem~\ref{T:MainTheorem}, as stated in the Introduction.
This result follows from the observation that when $\lambda \leq 1/\sqrt{2}$, we have $m^* \geq 1$ and $1/\lambda^2 \geq 2$, so the convergence behaviour is eventually dominated by wave number regime I and the errors (in the maximum norm) will be $O(h^2)$.
However, if $\lambda > 1/\sqrt{2}$, then $m^*=1$ and $1/\lambda^2< 2$, so the numerical error of the scheme is dominated by wave number regime IV. The scheme still converges but the errors will now be  $O(h^{1/\lambda^2}/\sqrt{\log(1/h)})$. The logarithmic factor in this expression is slow growing and the error will be close to $O(h^{1/\lambda^2})$ which is worse than $O(h^2)$.

So our result is that we always get convergence in the transformed numerical scheme whereas, without the time change, the errors eventually increase as $h \rightarrow 0$, with $\lambda$ controlling how small $h$ has to be for the divergence to manifest itself in practice.

\section{Numerical results and complexity considerations}
\label{sec:numerical}

Numerical experiments have been performed to investigate the behaviour of the time-changed scheme used to solve the heat equation with delta-function initial conditions. The analysis used $T=1$ and the space domain was truncated at $x=\pm 10$. The numerical scheme was run with a largest value of $k = 0.01$, i.e., with 100 time steps in the transformed coordinates. The grid was repeatedly refined by dividing  both $h$ and $k$ by 2 and the calculations were stopped when the number of time steps reached 3200.\\

The error (in the maximum norm) of the scheme was analysed by comparing the numerical results with the exact solution and the slope of the log-log plot of the error against $h$ was calculated. The slope was then plotted against the value of $\lambda$ in Figure 1, along with the $h$ exponent for the error derived from the analysis above, where we ignore the effect of the $\log(\sqrt{1/h})$ term on the behaviour of the error in wave number regime IV and just use the expression $\min(2,1/\lambda^2)$.

\begin{figure}[htb]
\caption{Convergence order for the time-changed scheme as a function of $\lambda$}
\includegraphics[width=5.5in,height=2.6in]{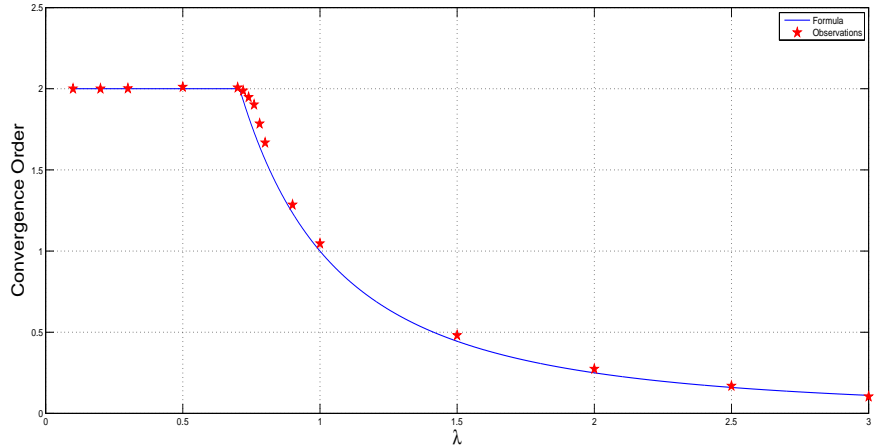}
\label{fig:ConvergenceOrder200912}
\end{figure}
We see a good match between theory and experiment with the larger divergences being seen where $\lambda$ is closest to the critical value of $1/\sqrt{2}$.
Figures 2 and 3 show the log-log plots of the errors versus $h$ for the cases $\lambda=0.5$ and $\lambda=1.0$, either side of $1/\sqrt{2}$.

\begin{figure}[!]
\caption{Error plot for the time-changed scheme with $\lambda = 0.5$ }
\includegraphics[width=5.5in,height=2.5in]{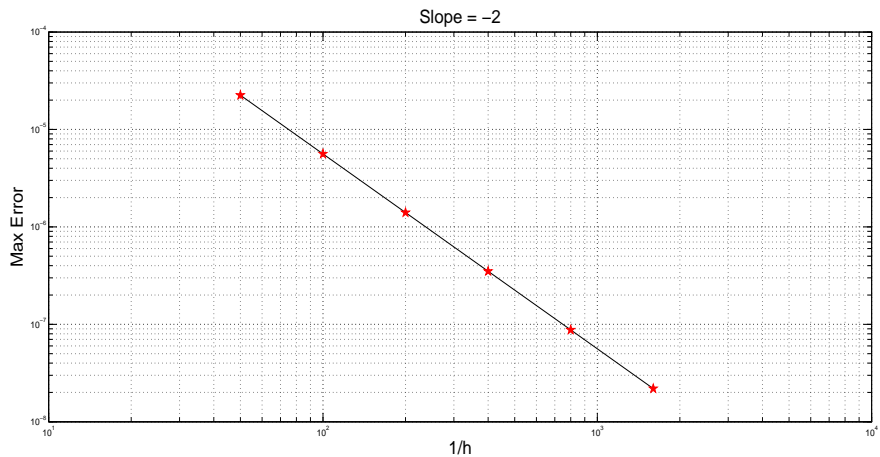}
\label{fig:LambdaHalf200912}
\end{figure}

\begin{figure}[!]
\caption{Error plot for the time-changed scheme with $\lambda = 1.0$ }
\includegraphics[width=5.5in,height=2.5in]{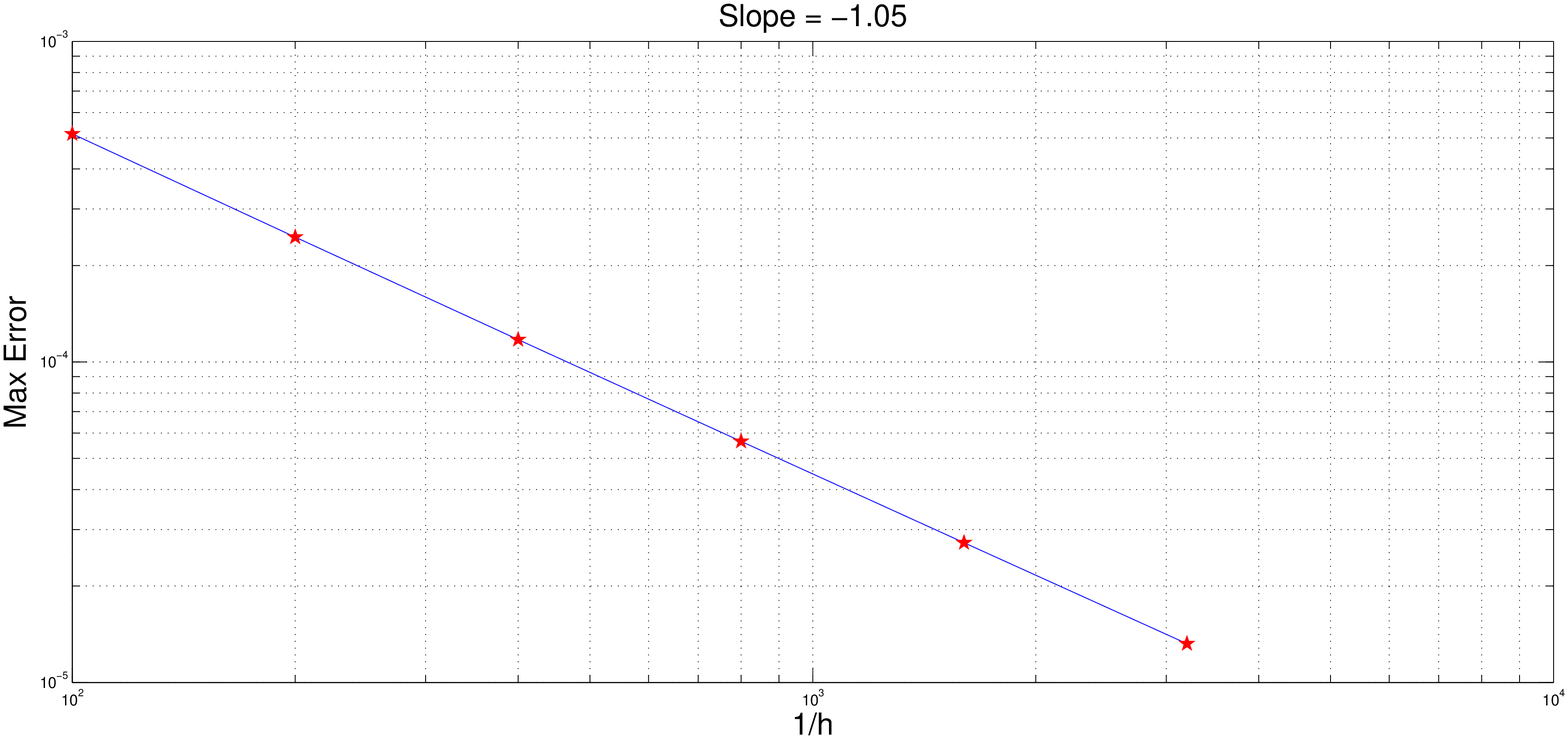}
\label{fig:LambdaOne200912}
\end{figure}


It is not possible to determine experimentally the $h$ exponent with complete accuracy for two reasons.

Firstly, as $h \rightarrow 0$, the time required for each solution increases rapidly so, for comparison purposes, all the calculations were terminated when the number of time steps reached 3200. As the overall error of the scheme is a mixture of errors with different $h$ exponents, the asymptotically dominating $h$ exponents will not be reached because the calculations are terminated early. In particular, for $\lambda$ slightly above $1/\sqrt{2}$, the contribution to the error from the highest wave number regime will be limited and lower values of $h$ need to be reached before this error clearly dominates the other errors. Note that our analysis does not identify the weighting factors for the errors of different order in $h$.

Secondly, the additional term, $\sqrt{\log(  1/h  )}$, in the error behaviour in wave number regime IV will distort the error plot and tend to increase the apparent $h$ exponent to some extent.

\bigskip

We have also carried out a comparison of the error performance between the time change scheme and the Rannacher scheme. 
For $\lambda \leq 1/\sqrt{2}$, both schemes will show quadratic convergence and, for small values of $h$, we can write

\[\widehat{U}^N_R(s)-\widehat{u}(s,1)=\widehat{u}(s,1)\left(\frac{1}{24}s^4+\frac{1}{8}\lambda^2 s^4-\frac{1}{96}\lambda^2 s^6\right)h^2(1+o(1))\]

\noindent 
and

\[\widehat{U}^N_{TC}(s)-\widehat{u}(s,1)=\widehat{u}(s,1)\left(\frac{1}{24}s^4+\frac{1}{8}\lambda^2 s^4-\frac{1}{48}\lambda^2 s^6\right)h^2(1+o(1)),\]

\noindent where $\widehat{U}^N_R(s)$ and $\widehat{U}^N_{TC}(s)$ are the Fourier transforms for the Rannacher and time-changed scheme in wave number regime I.

For $m \geq 1$, the $m$th derivative of the cumulative Normal distribution has the Fourier transform given by

\[\widehat{N^{(m)}}(s)=(is)^{(m-1)}e^{-s^2/2}\]  

\noindent so that the Fourier inverse of $(is)^{(m-1)} e^{-s^2/2}$ is $\frac{1}{\sqrt{2\pi}}N^{(m)}(x)      $. We can then estimate the ratio of the errors of the two schemes (at $x=0$) as

\[ \frac{E_R}{E_{TC}}=\frac{3\left(\frac{1}{24}+\frac{1}{8}\lambda^2\right)-15\frac{1}{96}\lambda^2}{3\left(\frac{1}{24}+\frac{1}{8}\lambda^2 \right)-15\frac{1}{48}\lambda^2}\]\

\noindent by evaluating $N^{(5)}(0)=3$ and $N^{(7)}(0)=-15$ from

\[N^{(5)}(x)=\frac{1}{\sqrt{2\pi}}(x^4-6x^2+3)e^{-x^2/2}\]

\noindent and

\[N^{(7)}(x)=\frac{1}{\sqrt{2\pi}}(-x^6-15x^4+45x^2-15)e^{-x^2/2}.\]

The ratio ${E_R}/{E_{TC}}$ has been plotted as a function of $\lambda$ in Figure \ref{fig:A} and the results compared to the empirical observations. So for values of $\lambda$ close to the critical value, there is a modest reduction in the error of the numerical scheme due to the time change method. 

Note that for $\lambda \geq 1/\sqrt{2}$ the errors will no longer be comparable as the order of convergence of the time-changed scheme is no longer quadratic.

\begin{figure}[htb]
\caption{Plot of ratio of Rannacher to time change error}
\includegraphics[width=5.5in,height=2.6in]{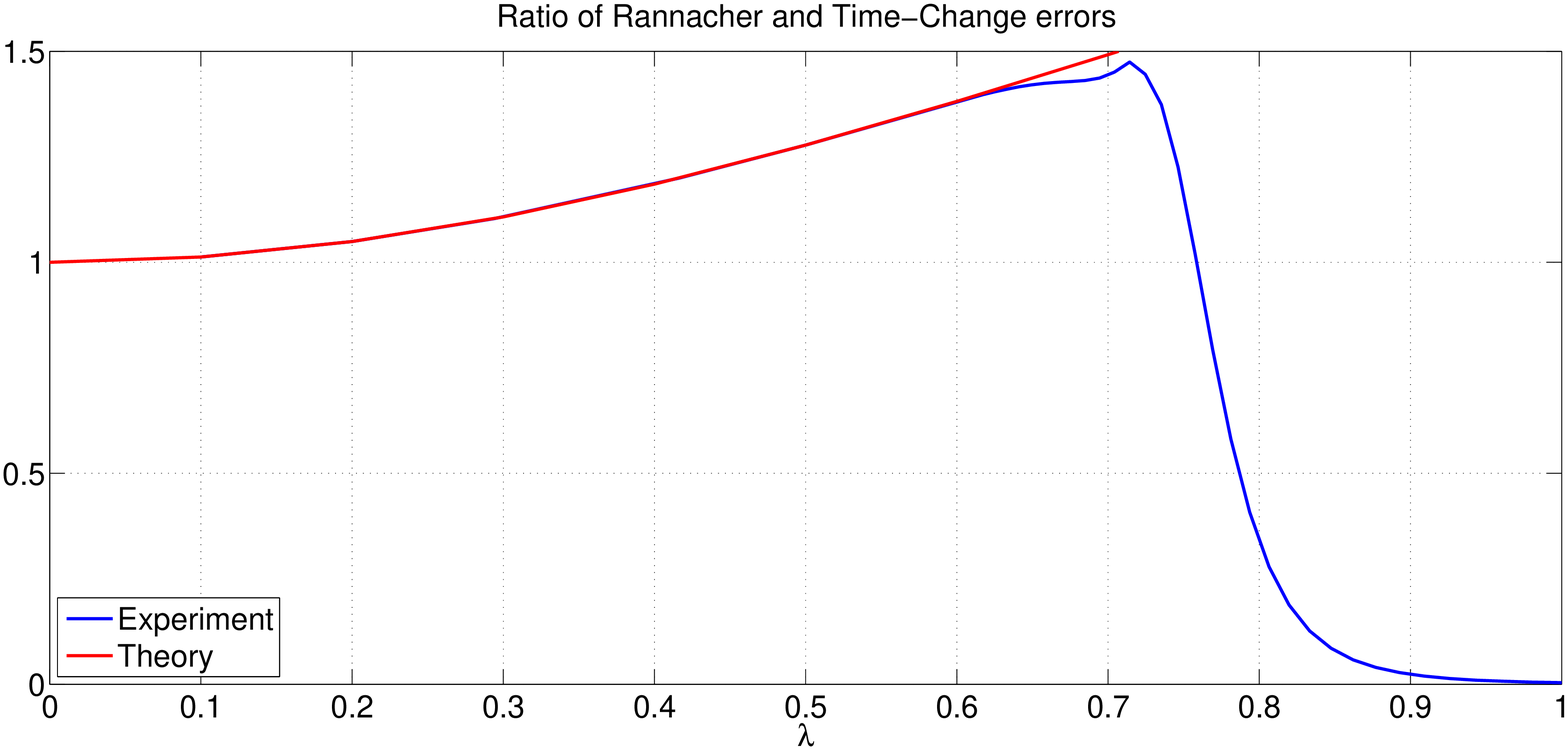}
\label{fig:A}
\end{figure}

\bigskip

We are also interested in the computational efficiency of the scheme, so we have investigated the relative performance of the two schemes when both are subject to a constraint on the  computational cost, $C$. We can write $C\sim c(1/k)(1/h)$ for some constant $c$ and we then can approximate the errors as 
\begin{eqnarray*}
E_R &=& 3(1/24+\lambda^2/8)h^2-(15\lambda^2/96)h^2\\
&=& (1/8)h^2+(21/96)k^2
\end{eqnarray*}
\noindent and similarly
\begin{eqnarray*}
E_{TC}&=&(1/8)h^2+(3/96)k^2.
\end{eqnarray*}
Then
\begin{eqnarray*}
E_R&=&hk((1/8)(h/k)+(21/96)(k/h))\\
&=&(c/C)((1/8)(1/\lambda)+(21/96) \lambda)
\end{eqnarray*}
\noindent and similarly
\begin{eqnarray*}
E_{TC}&=&(c/C)((1/8)(1/\lambda)+(3/48) \lambda).
\end{eqnarray*}

The errors $E_R$ and $E_{TC}$ will be separately minimised as functions of $\lambda$ at\\
\[\lambda^*_R=\sqrt{(1/8)(96/21)}\sim0.756\]

and

\[\lambda^*_{TC}=\sqrt{(1/8)(48/3)}=\sqrt{2}.\]

\begin{figure}[htb]
\caption{Plot of ratio of Rannacher to time change optimal error at the same cost}
\includegraphics[width=5.5in,height=2.6in]{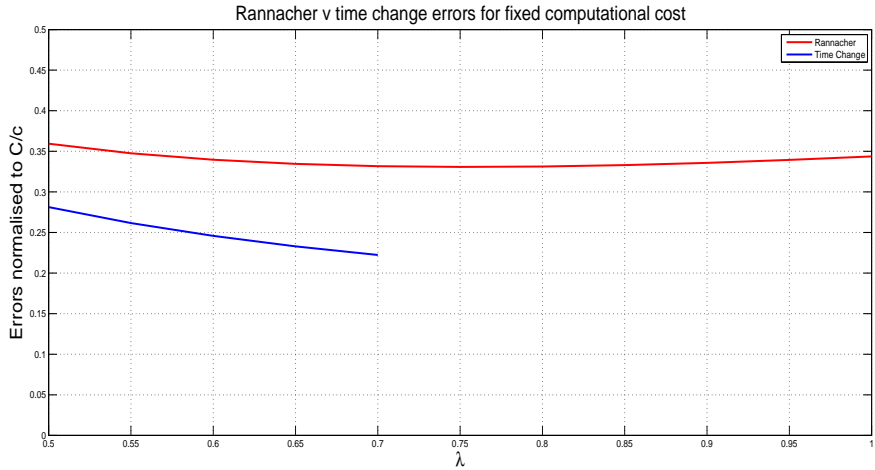}
\label{fig:B}
\end{figure}

Note that for the time change the analysis is only valid for $\lambda \leq 1/\sqrt{2}$, so the minimum occurs at $1/\sqrt{2}$.\\

\indent The plot (Figure \ref{fig:B}) of the errors for $\lambda \leq 1/\sqrt{2}$, with the fixed computational cost, shows that the time change scheme has the lowest achievable error for the value of $\lambda = 1/\sqrt{2}$ and this error is, in turn, better than the error of Rannacher scheme at the value $\lambda^*_R$.\\

\section{The time change applied to the Black-Scholes equation}~\label{S:BSApp}

In this section, we describe numerical results obtained by applying the time change to the Black-Scholes equation

\begin{eqnarray}
\label{bspde}
\frac{\partial V}{\partial t}+\frac{1}{2}\sigma^2S^2\frac{\partial^2 V}{\partial S^2}+rS\frac{\partial V}{\partial S}-rV=0, \quad S>0, t \in (0,T].
\end{eqnarray}

This is the equation that can be used to calculate, for example, the price, $V$, of a European option. Here, $S$ is the price of the underlying, $\sigma$ is the volatility, $r$ is the risk-free rate, and $t$ is time. For a call option, the terminal  condition is given by the payoff, i.e., $\max(S-K,0)$, where $K$ is the strike of the option. Practically important sensitivities of the call price to the initial asset price, used in risk management, are given by the delta, $\Delta$, and gamma, $\Gamma$, defined as

\[\Delta = \frac{\partial V}{\partial S} \]

\noindent and

\[\Gamma = \frac{\partial^2 V}{\partial S^2} \]

\noindent respectively.

Our investigation of convergence for the heat equation was based on Fourier analysis that relies on the PDE having constant coefficients whereas the Black-Scholes PDE used to price a European option does not have constant coefficients and so our analysis is not directly applicable in this case. 

\indent Nevertheless we have implemented the natural time change method to study the convergence behaviour of the gamma calculated with the modified numerical scheme. 

\indent In this context the natural analogue to the solution of the heat equation with Dirac delta initial conditions would be the PDE satisfied by the gamma,

\[\frac{\partial \Gamma}{\partial t}+\frac{1}{2}\sigma^2S^2\frac{\partial^2\Gamma}{\partial S^2}+(rS+2\sigma^2)\frac{\partial \Gamma}{\partial S}+(r+\sigma^2)\Gamma =0,\] 

\noindent where the terminal condition for this equation is given by the second derivative of the payoff with respect to $S$, i.e., $\Gamma(S,T) = \delta(S-K)$. Here we have also assumed that $\sigma$ and $r$ are constant.

\indent This PDE could be used to test the performance of the time change approach but a more useful approach (as it reflects industry practice) is to apply the time change directly to the Black-Scholes equation and then use finite differences to derive the values for the gamma from the calculated option values. (Note that the scheme was also tested with the gamma PDE and gave identical results).

We first note that if we can write our PDE in the form

\[\frac{\partial V}{\partial \tau}- \mathcal{L} V = 0,\]

\noindent where $\tau=T-t$ is the time to maturity and where the spatial differential operator $\mathcal{L}$ has time-independent coefficients, then the transformed equation for $\tilde{t} = \sqrt{\tau}$ is simply

\[\frac{\partial V}{\partial \tilde t} - 2\:\tilde t \:\mathcal{L} V = 0.\]

\noindent If the coefficients in $\mathcal{L}$ are time-dependent, then these need to be rewritten in terms of $\tilde t$. This simple change to the equation results in very minimal changes to the solver code. The calculation done for a single time step in the original variables can be written as

\[(I+L)V^{n+1}=(I-L)V^{n},\]
\noindent where $L$ is a matrix dependent on $k$, $h$, $S$, $\sigma$ and $r$ while $V^{n}$ is the solution vector at time step $n$. In the case of constant coefficients, and in the time-changed variables, this calculation simply becomes
\[(I+ 2 \:\tilde t_{n+1}L)\widetilde{V}^{n+1}=(I - 2 \:\tilde t_{n}L) \widetilde{V}^{n},\]
where $\widetilde{V}^n$ serves as numerical approximation to $\widetilde{V}(\cdot,\tilde{t}_n) = V(\cdot,T-\tilde{t}_n\,\!^2)$.
In particular, the transformed Black-Scholes equation is
\[
\label{bstrans}
\frac{\partial \widetilde{V}}{\partial \tilde t} - 2\:\tilde t\:\left(\frac{1}{2}\sigma^2S^2\frac{\partial^2\widetilde{V}}{\partial S^2}+rS\frac{\partial \widetilde{V}}{\partial S}-r\widetilde{V}\right)=0
\]

\noindent with the payoff function as the initial condition.

The transformed equation for $V$ was solved using the Crank-Nicolson scheme and the gamma calculated by finite differences. The errors in gamma were computed from the analytic values and the convergence rate of the numerical error as a function of $\lambda$ was analysed.

Although the scheme does not have the constant coefficients required for the application of Fourier analysis, the errors do behave in a very similar way to those for the heat equation (see Figure 4).

\begin{figure}[htb]
\caption{Convergence order for the gamma calculated from the time-changed Black-Scholes scheme as a function of $\lambda$}
\includegraphics[width=5.5in,height=2.5in]{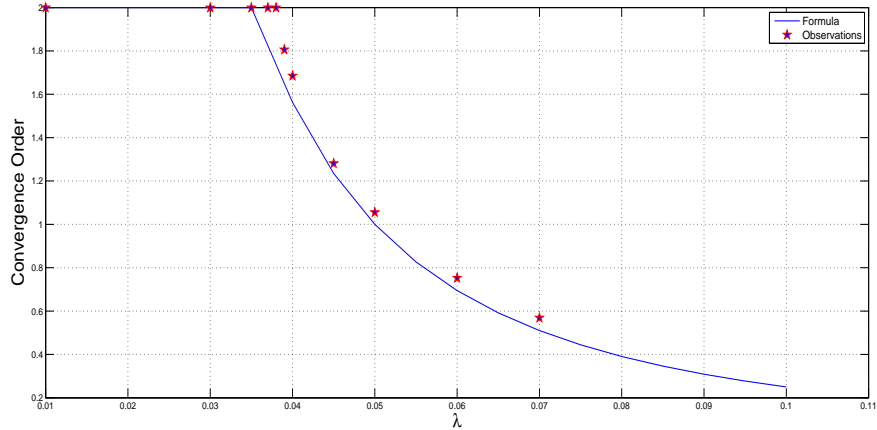}
\label{fig:BSConvergenceOrder200912}
\end{figure}

The volatility, $\sigma$, was set to 20$\%$ and the strike, $K$, was set to 100. The risk-free rate, $r$, was set to 5$\%$. The results of the analysis showed that there was a critical value of $\lambda$ below which the error in the maximum norm was $O(h^2)$ and above which the $h$ exponent declined in the same manner as for the heat equation. Also shown in Figure 4 is the plot of $\min(2,1/(\sigma^2 K^2 \lambda^2))$ and it is clear that this gives a reasonable description of the behaviour of the $h$ exponent although it is not quite as good as with the heat equation itself. 

\indent We can justify the critical value of $\lambda$ by noting that, in the Black-Scholes equation, the factor $\frac{1}{2}\sigma^2S^2$ plays the same role as the factor $\frac{1}{2}$ in the heat equation and we also observe that if we had written the heat equation in the form $u_t= \frac{1}{2}D \, u_{xx}$, for some constant diffusivity $D$, we would have found the critical value of $\lambda$ to be $1/\sqrt{2D}$.
\newline
\indent Typically, the worst behaviour of a numerical scheme for the Black-Scholes equation with delta function initial conditions is seen at times close to maturity and in the vicinity of the strike and so the strike is an appropriate parameter to use to represent the coefficient of $\frac{\partial^2V}{\partial S^2}$ which then determines the critical value of $\lambda$.


Even without the precise results available from Fourier analysis, we can still attempt a partial explanation of how the time change improves the convergence of the Crank-Nicolson numerical scheme. One consideration is that the truncation error of the numerical scheme has a leading order term proportional to the third derivative of the European option price and the behaviour of this derivative is much improved by the time change. 

The theta, $\Theta_C$, of the call option is given by (see \citet{HAUG}, p.64)

\[\Theta_C=\frac{\partial C}{\partial t}=\frac{Sn(d_1)\sigma}{2\sqrt{T-t}}-rKe^{-r(T-t)}N(d_2), \]

\noindent where $C$ is the solution to (\ref{bspde}) with $C(S,T) = \max(S-K,0)$,
$N(x)$ is the cumulative Normal(0,1) distribution, $n(x)$ is its first derivative (the Gaussian function) and

\[d_{1,2}=\frac{(       \log(S/K)  +       (r\pm\frac{1}{2}\sigma^2)(T-t)         )   }{       \sigma\sqrt{T-t}    }.\]

Clearly $\Theta_C$ will become singular at the strike as $t \rightarrow T$ and any higher derivatives will also be singular there.

However, if we change the time variable to $\tilde t = \sqrt{T-t}$, we find that

\[\tilde \Theta_C=\frac{\partial C}{\partial \tilde t}=-2\tilde t\frac{\partial C}{\partial t}=-2\tilde t\left(\frac{Sn(d_1)\sigma}{2\tilde t}-rKe^{-r\tilde t^2}N(d_2)\right)\]
\[=-Sn(d_1)\sigma+2\tilde t rKe^{-r\tilde t^2}N(d_2)\]

\noindent and $\tilde \Theta_C$ is then not singular at $\tilde t=0$.

Carrying on to the second `time' derivative we see that

\[\frac{\partial^2 C}{\partial \tilde t^2}= -Sn'(d_1)\sigma \frac{\partial d_1}{\partial \tilde t}+2rKe^{-r\tilde t^2}(1-2\:\tilde t\: r)N(d_2)-2\:\tilde t \:r Ke^{-r\tilde t^2}n(d_2)\frac{\partial d_2}{\partial \tilde t},\]

\noindent where we also have

\[\frac{\partial d_{1,2}}{\partial \tilde t}=-\frac{\log{\frac{S}{K}}}{\sigma \:\tilde t^2}+\frac{r \pm \frac{\sigma^2}{2}}{\sigma}\]

\noindent so there are potentially singular terms to consider, i.e., those featuring inverse powers of $\tilde t$. 

However, if $S \neq K$ then both $d_1$ and $d_2 \rightarrow \pm\infty$ as $\tilde t \rightarrow 0$ so that the exponential factors in $n(d_i)=(1/\sqrt{2\pi})\exp{(-d_i^2/2)}$ will eventually overwhelm the inverse powers of $\tilde t$.

Finally, if $S=K$ the term $\log(S/K)$ vanishes and the derivative is non-singular at the strike. 

It is clear that this analysis will extend to the higher derivatives of $C$ because the negative powers of $\tilde t$ will now only arise from differentiating $N$ and will then only appear as  multipliers of the Gaussian function. These powers will only be present for $S \neq K$ where the Gaussian function will decay rapidly and overwhelm the inverse powers. So the higher derivatives will not become singular as $\tilde t \rightarrow 0$.
It is also clear that the factor $\tilde{t}$ in (\ref{bstrans}) causes the singular spatial truncation error, which contains the fourth derivative of $V$ with respect to $S$, to become more well-behaved for small $\tilde{t}$ in the vicinity of the strike.
So the time change has, to some extent, 
regularised the behaviour of the truncation error.
 
Although it is not entirely clear how to deduce convergence properties of the transformed numerical scheme directly from this, and particularly convergence of derivatives of the solution,
it seems plausible that the improved regularity of the truncation error is advantageous when using initial conditions that are not smooth.

\section{American options}~\label{S:AmericPut}

\citet{FV} investigate the convergence properties of Crank-Nicolson time-stepping when used to calculate the price of an American put by a penalty method.
The penalised Black-Scholes equation is given by
\[\frac{\partial V}{\partial t}+\frac{1}{2}\sigma^2S^2\frac{\partial^2 V}{\partial S^2}+rS\frac{\partial V}{\partial S}-rV + \rho \max( \max(K-S,0)-V,0)=0,\]
\noindent
where $\rho>0$ is a penalty parameter. 
This non-linear equation is solved with the terminal condition set equal to the option payoff. The penalty term comes into play when the solution violates the requirement that $V$ should not fall below the payoff. As the penalty parameter $\rho$ increases, the solution of the equation converges to the value for the American option satisfying
\[
\max\left(\frac{\partial V}{\partial t}+\frac{1}{2}\sigma^2S^2\frac{\partial^2 V}{\partial S^2}+rS\frac{\partial V}{\partial S}-rV,
\max(K-S,0)-V\right) = 0.
\]
For the present purposes, we are interested solely in the convergence of the Crank-Nicolson scheme and choose $\rho$ large enough to make the penalty solution numerically indistinguishable from its limit.

\citet{FV} implemented their solver using the Crank-Nicolson scheme with Rannacher start up and demonstrated that as the time step was reduced with the ratio of time step to space step held constant, the error in the calculated at the money (ATM) value of the American put tended to $O(h^{\frac{3}{2}})$ behaviour. Their analysis of the problem suggests that non-uniform time-stepping might restore second order convergence so they went on to describe and implement an adaptive time-stepping procedure that did indeed yield second order convergence for the option value. 

Their results suggest that the simple time change method described above might also be able to restore second order convergence. 

\begin{table}[thb]
\centering
\begin{tabular}{|l|l|l|l|l|}
\hline
\multicolumn{5}{|c|}{Ratios of successive differences} \\
\hline
\multicolumn{5}{|c|}{$\lambda$ = 0.0125} \\
\hline
$M$ & $N$ & Value & Delta & Gamma \\ \hline
\ 3200 & 640 & 3.98 & 3.92 & 3.91 \\
\ 6400 & 1280 & 3.97 & 3.97 & 3.94 \\
\ 12800 & 2560 & 3.97 & 3.97 & 3.96 \\
\ 25600 & 5120 & 3.99 & 3.97 & 3.99 \\ \hline
\end{tabular}\\
\medskip
\begin{tabular}{|l|l|l|l|l|}
\hline
\multicolumn{5}{|c|}{Ratios of successive differences} \\
\hline
\multicolumn{5}{|c|}{$\lambda$ = 0.0250} \\
\hline
$M$ & $N$ & Value & Delta & Gamma \\ \hline
\ 3200 & 320 & 3.84 & 4.01 & 3.80 \\
\ 6400 & 640 & 3.99 & 3.93 & 3.96 \\
\ 12800 & 1280 & 3.94 & 3.93 & 3.89 \\
\ 25600 & 2560 & 3.96 & 3.97 & 3.98 \\ \hline
\end{tabular}\\
\medskip
\begin{tabular}{|l|l|l|l|l|}
\hline
\multicolumn{5}{|c|}{Ratios of successive differences} \\
\hline
\multicolumn{5}{|c|}{$\lambda$ = 0.0500} \\
\hline
$M$ & $N$ & Value & Delta & Gamma \\ \hline
\ 3200 & 160 & 3.85 & 3.65 & 2.07 \\
\ 6400 & 320 & 3.78 & 3.87 & 2.08 \\
\ 12800 & 640 & 3.89 & 3.87 & 2.09 \\
\ 25600 & 1280 & 3.89 & 3.90 & 2.09 \\ \hline
\end{tabular}\\
\medskip
\caption{Ratios of successive differences between numerical solutions for successive refinements for $V$, $\Delta$ and $\Gamma$; $M$ the number of steps in $S$, $N$ the number of steps in $t$; $\sigma=0.20$, $r=0.05$,
$\rho=10^6$}
\label{tab:Table1}

\end{table}

The penalty code was modified to incorporate the time change and numerical results were generated. Second order convergence was observed for the option value and delta for the ATM option and it was also observed that, for sufficiently small $\lambda$, the gamma showed second order convergence. Table ~\ref{tab:Table1} shows the ratios of successive differences for mesh refinements with $\lambda$ equal to 0.0125, 0.025 and 0.05. As with the European option discussed above, and with the strike, $K=100$, and the volatility, $\sigma=20\%$, the critical value of $\lambda$ appears to be in the vicinity of $1/(\sigma K \sqrt{2})=0.035$.

It was also observed that the plots of the calculated gamma revealed instability around the exercise boundary, behaviour which was also seen in \citet{FV}, where the authors found that their adaptive time-stepping approach resulted in instability of the gamma at the exercise boundary unless fully-implicit time-stepping was used. This behaviour is a separate phenomenon independent of the short-time asymptotics addressed here.


In the previous section we suggested that the improved convergence properties of the Crank-Nicolson scheme might be related to improved regularity of the third time derivative of the value function of the European put in the vicinity of the strike and close to maturity. 

We would like to be able to perform a similar analysis for the American put but, of course, no simple closed form expression for the value function is known. However, \citet{MA} do present an asymptotic expansion for the value function which can be decomposed into the value function of the European put plus the sum of an asymptotic series representing the early exercise premium.

In \citet{MA}, some changes of variable are used to define the final form of the expansion. We write $S=Ke^x$ 
and $\xi=x/(2\sqrt{\tau_s})$, where $\tau_s$ is the rescaled time to maturity, equal to $2(T-t)/\sigma^2$. The transformed value function is then given by
\begin{equation}
v(x,\tau_s)=\frac{P(S,t)+S-K}{K},
\end{equation}
where $P(S,t)$ is the value of the American put as a function of the original variables. We can then write that 

\begin{eqnarray}
\label{americanexp}
v(x,\tau_s)=v^E(x,\tau_s)+\sum_{n=2}^{\infty} \sum_{m=1}^{\infty}\tau_s^{n/2}(-\log \tau_s)^{-m}F^{m}_{n}(\xi),
\end{eqnarray}

\noindent where $F^{m}_{n}$ are functions of the similarity variable $\xi$ and $v^E(x,\tau_s)$ is the value of the European put expressed in the new variables.

In \citet{MA}, only a few of the functions $F^{m}_{n}$ are fully identified and discussion of those is restricted to $m=1$. Those with $n \geq 0$ and $m=1$ are, in principle, tractable but only partial results are known for the cases where $m>1$. This limits our ability to draw conclusions from the behaviour of the derivatives of the early exercise premium. Nevertheless, we will illustrate some aspects of the behaviour of the derivatives by describing a single term of the early exercise premium and its form and behaviour after the time change.

For example we find that the leading order coefficient of the early exercise premium is given by 
\[ 
F^{1}_{2}(\xi) = \mu(\xi \exp{(-\xi^2/2) }/\sqrt{\pi} + (2\xi^2+1)\mathrm{erfc}(\xi)) \]

\noindent for some constant $\mu$.

After the time change this early exercise term in (\ref{americanexp}) is
\[
G = \mu \tilde t^2(-2\log \tilde t)^{-1}F^{1}_{2}(\eta),
\]
where $\eta=x/(2\tilde t)$. At the strike, we have 
\begin{eqnarray*}
\frac{\partial G}{\partial \tilde t} &=& \mu(2\tilde t(-2\log \tilde t)^{-1} -2\tilde t(-2\log \tilde t)^{-2}),  \\
\frac{\partial^2 G}{\partial \tilde t^2} &=& \mu(2(-2\log \tilde t)^{-1} +4(-2\log \tilde t)^{-2})+8(-2\log \tilde t)^{-3})
\end{eqnarray*}
\noindent and
\begin{eqnarray*}
\frac{\partial^3 G}{\partial \tilde t^3}=\frac{\mu}{\tilde t}(4(-2\log \tilde t)^{-2} +16(-2\log \tilde t)^{-3})+48(-2\log \tilde t)^{-4}).
\end{eqnarray*}

Acknowledging that this calculation only addresses a single term of the expansion of the early exercise premium we see that there is an improvement in the regularity of the derivatives of the put but unfortunately this does not obviously extend to the third derivative terms such as ${\partial^3 G}/{\partial \tilde t^3}$.

So, although we are encouraged by the empirical performance of the time change method applied to the American put, further analysis will be required to explain fully the improvement in convergence.   

\section{Conclusions}
\label{sec:concl}

In this paper we have presented the analysis of a simple and natural change of time variable that improves the convergence behaviour of the Crank-Nicolson scheme. When applied to the solution of the heat equation with Dirac delta initial conditions, the numerical solution of the time changed PDE is always convergent with rate of convergence determined by the ratio, $\lambda$, of time step to space step (held constant during grid refinement). 
\newline
\indent This behaviour contrasts strongly to that of the Crank-Nicolson scheme when applied to the original PDE, which is always divergent and where $\lambda$ controls how small the time step must be before the divergence appears. A proof of the behaviour of the time-changed scheme is given and numerical results presented to support the theoretical analysis.

Although the Fourier analysis used to prove this result cannot be applied directly to the Black-Scholes equation as it does not have constant coefficients, numerical experiments for the price and Greeks of a European call indicate that the time change method also leads to a convergent Crank-Nicolson scheme for this problem, without the need to introduce Rannacher steps, and quadratic convergence can be obtained if $\lambda$ is chosen appropriately.

The change of convergence order at the value of $\lambda$ which gives optimal complexity for given error tolerance is in line with the sensitivity of the error to $\lambda$ observed in \citet{GC2}.
The present analysis does not support the conjecture made in \citet{GC2}, however, that the convergence of the Rannacher scheme with non-uniform (square-root) time-stepping is never of second order, the argument being that for small initial time-steps the smoothing effect diminishes. In fact, we see here that under square-root time smoothing is not necessary at all.

The experimentally observed second order convergence for the time-changed scheme extends to the computation of American option values which satisfy a linear complementarity problem with singular behaviour of the free boundary.
These results are in line with those of \citet{FV} for a time-adaptive Crank-Nicolson scheme with Rannacher start-up.

It can be seen as an advantage of the time-changed scheme proposed here that it provides a remedy for both the reduced convergence order of the Crank-Nicolson scheme for European and American option values and for the divergence of their sensitivities by a single, simple modification.


Future work in this area will investigate other applications of the time variable change to option pricing and will consider whether the method offers efficiency gains over the standard approaches.

\clearpage

\appendix

\section*{Appendix A.\ Proof of Lemma \ref{L:LemmaA}, (\ref{expapp}) and (\ref{expandapp})}
\label{app:lemma1}

We analyse the behaviour of the Fourier transform $\widehat{U}^N(s)$ in each of the four wavenumber regimes.
\subsection*{\textbf{Proof of (\ref{expapp})}:}
\label{app:exp}

We start from (\ref{EQ3}),
\[
\log(\widehat{U}^N(s))=\sum^{N-1}_{k=1} \left(\log(1-k\xi) - \log(1+k\xi)\right) - \log(1+N\xi).
\]
\noindent 
We seek to expand the logarithmic terms. We begin by first expanding $\xi=2 \lambda^2 \sin^2(\frac{sh}{2})$ in powers of $h$ to obtain 
\[
1-k\xi=1-\frac{k\lambda^2 s^2}{2!} h^2 +\frac{k\lambda^2 s^4}{4!} h^4 -\frac{k\lambda^2 s^6}{6!} h^6 + \frac{k\lambda^2 s^8}{8!}h^8 \theta_k
\equiv \sum_{i=0}^3 A_i h^{2i} \, + \, A_4 h^8,
\]\\
\noindent where $|\theta_k| \leq 1$. Setting $\delta=h^2$, we define
\begin{align}
\notag
g(\delta) \equiv 1-k\xi= \sum_{i=0}^3 A_i \delta^{i} \, + \, A_4 \delta^4
\end{align}
\noindent and
\begin{align}
\notag
f(\delta)=\log(1-k\xi)=\log{(g(\delta))}
\end{align}
\noindent and we can then expand $f(\delta)$ in terms of $\delta$ to obtain
\[\log(1-k\xi)=-\frac{k\lambda^2}{2!}s^2\delta+\frac{1}{2!}\left(\frac{2 k \lambda^2}{4!}-\frac{k^2 \lambda^4}{(2!)^2}\right)s^4\delta^2+\frac{1}{3!}\left(-\frac{6 k \lambda^2}{6!}-\frac{2k^3 \lambda^6}{(2!)^3}+\frac{6k^2 \lambda^4}{2!4!}\right)s^6\delta^3+Z_ks^8\delta^4\]
\noindent and we have to ensure that the remainder $Z_k$ is well-behaved. We do this by analysing the behaviour of 
\[f^{(4)}(\delta)=\frac{p(\delta)}{g(\delta)^4},\]\\
\noindent where\\
\[p(\delta)=g(\delta)^3 g^{(4)}(\delta)-4g(\delta)g'(\delta)g^{(3)}(\delta)+6(g'(\delta))^2g(\delta)g''(\delta)-3(g''(\delta))^2g(\delta)-6(g'(\delta))^4.\]\\
Then $g(\delta) \rightarrow 1$ as $h \rightarrow 0$ if we ensure that $s<h^{-m}$, where $m < \frac{1}{2}$, because, for example,
\[|A_1 \delta| = \frac{k \lambda^2 s^2 h^2}{2} < \frac{N\lambda^2 s^2 h^2}{2} = \frac{ \lambda s^2 h}{2} \rightarrow 0.\]
Then $h$ can be chosen small enough (or $N$ large enough) so that
\[|1+A_1 \delta +A_2 \delta^2 +A_3 \delta^3 +A_4 \delta^4| > \frac{1}{2}\]
\noindent so the denominator of $f^{(4)}(\delta)$ is bounded away from zero.
Next we can show that
\[|g^{(i)}(\delta)|<\alpha_i|A_i|\]

\noindent for some constants, $\alpha_i$ for $i=0,1,2,3$ and $4$, because of the constraint on $s$. We can finally conclude that\\
\[|f^{(4)}(\psi\delta)|\delta^4<(2\alpha_4|A_4|+32\alpha_1\alpha_3|A_1||A_3|+48\alpha_1^2 \alpha_2|A_1|^2 |A_2|+24\alpha_2^2|A_2|^2+96\alpha_1^4|A_1|^4)h^8.\]\\
and it follows that we can estimate asymptotically the remainder by just retaining the term proportional to $|A_1|^4 h^8$ which is, in turn, proportional to $k^4 \lambda^8s^8h^8$.\\
\newline
\indent By changing the sign of $k$, we can obtain a similar expression for $\log(1+k\xi)$ hence find that
\[\log(1-k\xi)-\log(1+k\xi)=-k \lambda^2 s^2 h^2 +\frac{1}{12}k \lambda^2 s^4 h^4-\frac{1}{360}k \lambda^2 s^6 h^6-\frac{1}{12}k^3 \lambda^6 s^6 h^6 +W_k k^4 s^8 h^8, \]\

\noindent where $|W_k|<K$ for some constant, $K$, and so

\begin{align}
\notag
\log(\widehat{U}^N(s))&=\left(-\lambda^2 s^2 h^2 +\frac{1}{12}\lambda^2 s^4 h^4-\frac{1}{360}\lambda^2 s^6 h^6\right)\left(\sum_{k=1}^{k=N-1}k\right)
-\frac{1}{12}\lambda^6 s^6 h^6 \left(\sum_{k=1}^{k=N-1}k^3\right)\\
\notag
&+\left(-\frac{\lambda^2 s^2 h^2}{2}+\frac{1}{24}\lambda^2 s^4 h^4-\frac{1}{720}\lambda^2 s^6 h^6\right)N
+\frac{1}{8}\lambda^4 s^4 h^4 N^2- \frac{1}{48}\lambda^4 s^6 h^6 N^2 \\
\notag
&- \frac{1}{24}\lambda^6 s^6 h^6 N^3  +O(s^8h^3),
\end{align}
\noindent because $(\sum_{k=1}^{k=N}k^4)\lambda^8s^8h^8=O(N^5\lambda^8s^8h^8)=O(\lambda^3s^8h^3)$.
So
\begin{align}
\notag
\log(\widehat{U}^N(s))&=\left(-\lambda^2 s^2 h^2 +\frac{1}{12}\lambda^2 s^4 h^4-\frac{1}{360}\lambda^2 s^6 h^6\right)\left(\frac{1}{2}(N-1)N+\frac{N}{2}\right)\\
\notag
&-\frac{1}{12}\lambda^6 s^6 h^6 \left(\left(\frac{1}{2}(N-1)N\right)^2+\frac{N^3}{2}\right)+O(s^8h^3)
\end{align}

\noindent as the remaining terms are dominated by $O(s^8h^3)$. Substituting $N=\frac{1}{h\lambda}$ we get

\begin{align}
\notag
\log(\widehat{U}^N(s))&=-\frac{1}{2}s^2 +\left(\frac{1}{24}s^4-\frac{1}{48}\lambda^2 s^6+\frac{1}{8}\lambda^2s^4\right)h^2-\left(\frac{1}{720}s^6+\frac{1}{48}\lambda^2 s^6+\frac{1}{48}\lambda^4 s^6\right)h^4+O(s^8h^3)
\end{align}

%

\noindent and hence

\[\log(\widehat{U}^N(s))=-\frac{1}{2}s^2 +\left(\frac{1}{24}s^4-\frac{1}{48}\lambda^2 s^6+\frac{1}{8}\lambda^2 s^4\right)h^2+O(s^8h^3),\]\

\noindent because the $s^6h^4$ term of the expansion is dominated by the term $O(s^8h^3)$.\\

\subsection*{\textbf{Proof of (\ref{expandapp})}:}
\label{app:expand}

We start from (\ref{secexp}),
\[
\log\big(\big|\widehat{U}^N(s)\big|\big)=\log(M(\xi,m^*))+\log\left\{\frac{(1-\frac{1}{\xi(m^*+1)})\ldots(1-\frac{1}{\xi(N-1)})} {(1+\frac{1}{\xi(m^*+1)})\ldots(1+\frac{1}{\xi (N-1)})}\right\}+\log((1+N\xi)^{-1}),
\]
\noindent and write
\[\log{\frac{(1-\frac{1}{\xi(m^*+1)})\ldots(1-\frac{1}{\xi(N-1)})} {(1+\frac{1}{\xi(m^*+1)})\ldots(1+\frac{1}{\xi (N-1)})}}=\sum_{k=m^*+1}^{k=N-1} \left(\log\left(1-\frac{1}{k\xi}\right) -\log\left(1+\frac{1}{k\xi}\right)\right).
\]

\noindent We can expand the logarithmic terms for $\xi\in[1/m^*,2\lambda^2]$ to get

\begin{align}
\notag
\sum_{k=m^*+1}^{k=N-1} \left(\log(1-\frac{1}{k\xi}) -\log(1+\frac{1}{k\xi})\right)&=\sum_{k=m^*+1}^{k=N-1}\left(\sum_{j=1}^{\infty}\frac{(-1)}{j(k\xi)^j}-\sum_{j=1}^{\infty}\frac{(-1)^{j+1}}{j(k\xi)^j}\right)\\
\notag
&=-\sum_{j=0}^{\infty}\frac{2}{(2j+1)\xi^{2j+1}}\left(\sum_{k=m^*+1}^{k=N-1}\frac{1}{k^{2j+1}}\right)
\end{align}

\noindent and, for any $N$, this series converges absolutely in $1/\xi$ as we have just rearranged a finite sum of absolutely convergent series. We can write this as

\begin{align}
\notag
\sum_{k=m^*+1}^{k=N-1} \left(\log(1-\frac{1}{k\xi}) -\log(1+\frac{1}{k\xi})\right)&=-\frac{2}{\xi}\left(\sum_{k=m^*+1}^{k=N-1}\frac{1}{k}\right)  -\sum_{j=1}^{\infty}\frac{2}{(2j+1)\xi^{2j+1}}\left(\sum_{k=m^*+1}^{k=N-1}\frac{1}{k^{2j+1}}\right)\\
\notag
&=-\frac{2}{\xi}\left(\sum_{k=m^*+1}^{k=N-1}\frac{1}{k}\right)+A(\xi,m^*)+B(\xi,N),
\end{align}

\noindent where\\

\[A(\xi,m^*)=-\sum_{j=1}^{\infty}\frac{2}{(2j+1)\xi^{2j+1}}\left(\sum_{k=m^*+1}^{\infty}\frac{1}{k^{2j+1}}\right) \]\\

\noindent and

\[B(\xi,N)=\sum_{j=1}^{\infty}\frac{2}{(2j+1)\xi^{2j+1}}\left(\sum_{k=N}^{\infty}\frac{1}{k^{2j+1}}\right), \]\\

\noindent and we must establish the convergence and size of the functions $A(\xi,m^*)$ and $B(\xi,N)$ (see below).\\

We also have\\

\begin{align}
\notag
-\log(1+N\xi)=-\log{N\xi}-\log{(1+{1}/{N\xi})}&=-\log{N\xi}-\sum_{j=1}^{\infty}\frac{(-1)^{j+1}}{j(N\xi)^j}\\
\notag
&=-\log{N\xi}-\frac{1}{N\xi}+C(\xi,N),
\end{align}

\noindent where

\[C(\xi,N)=-\sum_{j=2}^{\infty}\frac{(-1)^{j+1}}{j(N\xi)^j}\]

\noindent so that

\begin{align}
\notag\
\log\big(\big|\widehat{U}^N(s)\big|\big)&=\log(M(\xi,m^*))-\frac{2}{\xi}\left(\sum_{k=m^*+1}^{k=N-1}\frac{1}{k}\right)+A(\xi,m^*)+B(\xi,N)-\log{N\xi}-\frac{1}{N\xi}+C(\xi,N)\\
\notag\
&=\log(M(\xi,m^*))-\frac{2}{\xi}\left(\sum_{k=m^*+1}^{k=N}\frac{1}{k}\right)+A(\xi,m^*)+B(\xi,N)-\log{N\xi}+\frac{1}{N\xi}+C(\xi,N)\\
\notag\
&=P(\xi,m^*)-\frac{2}{\xi}\left(\sum_{k=1}^{k=N}\frac{1}{k}\right)+D(\xi,N)-\log{N}+O(N^{-1}),
\end{align}
\noindent where
\[P(\xi,m^*)=\log(M(\xi,m^*))-\frac{2}{\xi}\left(\sum_{k=1}^{k=m^*}\frac{1}{k}\right)+A(\xi,m^*)-\log{\xi}\]
and
\[D(\xi,N)=B(\xi,N)+C(\xi,N).\]

Then we can write

\begin{align}
\notag\
\log\big(\big|\widehat{U}^N(s)\big|\big)&=P(\xi,m^*)-\frac{2}{\xi}\left(\log{N}+\gamma+O(N^{-1})\right)+D(\xi,N)-\log{N}+O(N^{-1})\\
\notag\
&=Q(\xi,m^*)-\left(    \frac{2}{\xi}+1     \right) \log{N} + O(N^{-1})  +D(\xi,N),
\end{align}

\noindent where $\gamma$ is the Euler-Macheroni constant \citep[see][]{TT}, and where
\[
Q(\xi,m^*)=P(\xi,m^*)-\frac{2}{\xi}\gamma.
\]
Then using $N=1/(h\lambda)$ and assuming for now
$D(\xi,N)=O(N^{-1})$, which we will show later,
we get
\begin{align}
\notag\
\log\big(\big|\widehat{U}^N(s)\big|\big)&=R(\xi,m^*)+\left(\frac{2}{\xi}+1\right)\log{h}+O(h),
\end{align}
\noindent where 
\[
R(\xi,m^*)=Q(\xi,m^*)+\left(\frac{2}{\xi}+1\right)\log{ \lambda }.
\]\\
Then we finally have

\begin{align}
\notag\
\big|\widehat{U}^N(s))\big|&=S(\xi,m^*)h^{\left(\frac{2}{\xi}+1\right)}(1+O(h))\\
\end{align}

\noindent as we were required to prove. We note that, written out in full, we have
\[
S(\xi,m^*)= M(\xi,m^*) \lambda^{ (\frac{2}{\xi}+1)  }  \exp{  \left(  -\frac{2}{\xi} \left(\gamma+ \sum_{k=1}^{k=m^*}\frac{1}{k} \right)   \right)  }  e^{A(\xi,m^*)}
\]
\noindent and we note that $S(\xi,m^*)$ is continuous and hence bounded on $[1/m^*,2\lambda^2]$.\\

We still need to prove that $A(\xi,m^*)+B(\xi,N)$ is a valid rearrangement and proceed by proving that $A(\xi,m^*)$ is convergent. It will then follow that $B(\xi,N)$ is also convergent. In addition we will show that $D(\xi,N)=B(\xi,N)+C(\xi,N)=O(N^{-2})$.\\

By a result in \citet{Timofte}, we have, for $r \geq 1$,

\[\sum_{k=m^*+1}^{k=\infty}\frac{1}{k^{2r+1}}=\frac{1}{2r(m^*+\theta_{m^*})^{2r}}\]

\noindent for some $\theta_{m^*} \geq \frac{1}{2}$. 
%
%
Therefore

\begin{eqnarray*}
|A(\xi,m^*)| &<& \sum_{j=1}^{\infty}\frac{2}{(2j+1)}\frac{1}{\xi^{2j+1}}\frac{1}{2j(m^*+\frac{1}{2})^{2j}} \\
&=& \left(m^*+\frac{1}{2}\right)\sum_{j=1}^{\infty}\frac{1}{j(2j+1)}\frac{1}{(\xi(m^*+\frac{1}{2}))^{2j+1}} \\
&<& \frac{1}{3}\left(m^*+\frac{1}{2}\right)\eta^3\sum_{j=0}^{\infty}\eta^{2j},
\end{eqnarray*}

\noindent where $\eta=1/(\xi(m^*+\frac{1}{2}))<1$ and so $A(\xi,m^*)$ is convergent. It follows that $B(\xi,N)$ is convergent because $A(\xi,m^*)+B(\xi,N)$ is convergent.\\

Next we find out how $B(\xi,N)$ depends on $N$. We have

\begin{eqnarray*}
|B(\xi,N)| &<& 2\sum_{j=1}^{\infty}\frac{1}{(2j+1)}\frac{1}{\xi^{2j+1}}\frac{1}{2j(N-1+\frac{1}{2})^{2j}} \\
&=& 2N\sum_{j=1}^{\infty}\frac{1}{(2j+1)}\frac{1}{(N\xi)^{2j+1}}\frac{N^{2j}}{2j(N-\frac{1}{2})^{2j}},
\end{eqnarray*}
\noindent and as $N/(N-\frac{1}{2}) \leq 2$, we have

\[|B(\xi,N)|<2N\sum_{j=1}^{\infty}\frac{1}{(2j+1)}\frac{1}{(N\xi)^{2j+1}}\frac{2^{2j+1}}{4j} < \frac{1}{12}N \sum_{j=1}^{\infty}\left(\frac{2}{N\xi}\right)^{2j+1}.\]

\noindent Then because we have $2/\xi \leq 2m^*$, we find that for sufficiently large $N$, we have $2/N\xi<c<1$ for a constant $c$ and the geometric series converges and we find that
\[|B(\xi,N)| < \frac{1}{6}N\left(\frac{2}{N\xi}\right)^3\frac{1}{\left(1-\frac{2}{N\xi}\right)}=O(N^{-2})\]
and similarly for $C$.
It then follows that $D(\xi,N)=O(N^{-2})$. 

\section*{Appendix B.\ Proof of Lemma \ref{L:LemmaB}, (\ref{intapp})}
\label{app:int}


%
%
%
%
%
%
%
%
%
We want to show that 

\[I 
=\int^{2\lambda^2}_{\xi=\xi_{m^*}} \frac{ S(\xi,m^*)h^{(\frac{2}{\xi}-\frac{1}{\lambda^2})}}{  \lambda\sqrt{2}\sqrt{\xi}\sqrt{1-\frac{\xi}{2\lambda^2}} } \, d\xi =O\left(\frac{1}{ \sqrt{   \log{   \frac{1}{h}   }    }}\right),
\]
where $S(\cdot,m^*)$ is continuous.
We make the substitution $z^2=1-{\xi}/{2\lambda^2}$ to obtain 
%
%

\[
I=4\lambda^2\int_{z=0}^{z=\sqrt{B}}S^*(z,m^*)h^{\left(\frac{1}{\lambda^2}\frac{z^2}{1-z^2}\right)} \, dz,
\]

\noindent where $B=1-\frac{\xi_{m^*}}{2\lambda^2}$, and where $S^*(z,m^*)={S(\xi,m^*)}/\sqrt{1-z^2}$ is a continuous function on $[0,\sqrt{B}]$ as $B < 1$.

We want to show that this integral is concentrated around $z=0$ for small $h$ and write

\[I=I_1+I_2,\]

\noindent where

\[I_1=4\lambda^2\int_{z=0}^{z=\sqrt{A}}S^*(z,m^*)h^{\left(\frac{1}{\lambda^2}\frac{z^2}{1-z^2}\right)} \,dz\]

\noindent and

\[I_2=4\lambda^2\int_{z=\sqrt{A}}^{z=\sqrt{B}}S^*(z,m^*)h^{\left(\frac{1}{\lambda^2}\frac{z^2}{1-z^2}\right)} \, dz, \]

\noindent where $A$ is chosen so that, asymptotically, $I_1$ dominates $I_2$.

Firstly, we consider $I_2$. For $h < 1$, the second factor of the integrand
is decreasing in $z$, so
\[|I_2|<4\lambda^2 S^*_{max} h^{\left(\frac{1}{\lambda^2}\frac{A}{1-A}\right)},\]
\noindent where $S^*_{max}$ is a bound for $S^*$ on $[\xi_{m^*},2\lambda^2]$.
We now take
\[\frac{A}{1-A}=\frac{1}{\sqrt{\log{{1}/{h}}}}\]

\noindent and note that

\[\log{h^{\left(\frac{1}{\lambda^2}\frac{A}{1-A}\right)}}=-\frac{1}{\lambda^2}\frac{A}{1-A}\log{\frac{1}{h}}=-\frac{1}{\lambda^2}\sqrt{\log{\frac{1}{h}}}\]

\noindent so \[|I_2|<4\lambda^2 S^*_{max} e^{    -\frac{1}{\lambda^2} \sqrt{ \log{ \frac{1}{h} } }}
\rightarrow 0 \quad \mathrm{ as } \;\;\; h \rightarrow 0.\]

Now we want to show that

\[  \lim_{h \rightarrow 0} I_1 \sqrt{ \log{ \frac{1}{h} } } = \lim_{h \rightarrow 0} \left(4\lambda^2\sqrt{\log{\frac{1}{h}}} \int_{z=0}^{z=\sqrt{A}} S^*(z,m^*)h^{\left(\frac{1}{\lambda^2}\frac{z^2}{1-z^2}\right)} \, dz\right)\ =O(1)\]\\

\noindent and we note that

\[A=\frac{1}{1+\sqrt{\log{{1}/{h}}}} \rightarrow 0 \quad \mathrm{as} \;\;\: h \rightarrow 0.\]

\noindent We make a final substitution

\[\eta=\frac{z}{\lambda}\sqrt{\log{\frac{1}{h}}}, \]

\noindent to get

\[\lim_{h \rightarrow 0}I_1\sqrt{\log{\frac{1}{h}}}=4\lambda^2\lim_{h \rightarrow 0} \int_{\eta=0}^{\eta=\eta^*}S^*(\lambda \eta/\sqrt{\log{{1}/{h}}},m^*)\exp{\left(-\frac{\eta^2}{1-\lambda^2 \eta^2/\log{(1/h)}}\right)}\, d\eta,\]\

\noindent where the upper limit
\[
\eta^*=\frac{\sqrt{A}}{\lambda}\sqrt{\log{\frac{1}{h}}}=\frac{   \sqrt{   \log{   {1}/{h}   }    }    }{ \lambda  \sqrt{    1+\sqrt{    \log{{1}/{h}}   }     }      }
\]
tends to infinity as $h \rightarrow 0$.
As $S^*(\cdot,m^*)$ is bounded and continuous at 0 we see that

\[ \lim_{h \rightarrow 0} \int_{\eta=0}^{\eta=\eta^*}S^*(\lambda \eta/\sqrt{\log{{1}/{h}}},m^*)\exp{ \left(-\frac{\eta^2}{1-\lambda^2 \eta^2/\log{(1/h)}}\right) }\, d\eta=S^*(0,m^*)\int_{\eta=0}^{\eta=\infty}\exp{\left(-\eta^2\right)}\, d\eta\]\\

\noindent so we have $I_1\sqrt{\log{{1}/{h}}}=O(1)$ as $h \rightarrow 0$.\\
%
%
%

From this we see that $I_1=O({1}/{ \sqrt{   \log{   \frac{1}{h}   }    }})$ and we also have

\[\lim_{h \rightarrow 0}\left|\frac{I_2}{I_1}\right|=\lim_{h \rightarrow 0}\sqrt{   \log{   \frac{1}{h}   }    }e^{    -\frac{1}{\lambda^2} \sqrt{ \log{ \frac{1}{h} } }}= 0\]

\noindent so we have established that $I_1$ asymptotically dominates $I_2$ and so $I=O({1}/{ \sqrt{   \log{ {1}/{h}   }    }})$.


\end{document}